\documentclass[12pt]{amsart}

\usepackage{amsfonts}
\usepackage{amsthm}
\usepackage{amsmath}
\usepackage{mathtools}
\usepackage{amssymb}
\usepackage{dsfont}
\usepackage{amscd}
\usepackage[latin2]{inputenc}
\usepackage[mathscr]{eucal}
\usepackage{indentfirst}
\usepackage{graphicx}
\usepackage{graphics}
\usepackage{dirtytalk}
\usepackage{pict2e}
\usepackage{epic}
\usepackage{verbatim}
\usepackage{color,soul}
\usepackage{etoolbox}
\usepackage{enumitem}
\numberwithin{equation}{section}
\usepackage{epstopdf}
\usepackage[normalem]{ulem}
\usepackage{tikz-cd}
\usepackage[all]{xy}
 
\usepackage{hyperref}
\graphicspath{ {figures/} }
\usepackage{array}

\def\ep{\varepsilon}

\theoremstyle{plain}
\newtheorem{Th}{Theorem}[section]
\newtheorem{Lemma}[Th]{Lemma}

\theoremstyle{definition}

\newtheorem{Rem}[Th]{Remark}
\newtheorem{?}[Th]{Problem}
\newtheorem{Ex}[Th]{Example}

\newcommand{\interior}[1]{%
  {\kern0pt#1}^{\mathrm{o}}%
}
\newcommand{\R}{\mathbb{R}}
\newcommand{\N}{\mathbb{N}}

\newcommand{\abs}[1]{\left\lvert#1\right\rvert}
\newcommand{\norm}[1]{\left\lVert#1\right\rVert}
\newcommand{\op}[1]{\operatorname{#1}}
\renewcommand{\hat}[1]{\widehat{#1}}
\numberwithin{equation}{section}

\title{Hamiltonian perturbations in contact Floer homology}
\author{Igor Uljarevi\'c}
\email{igoru@matf.bg.ac.rs}
\address{Faculty of mathematics, University of Belgrade, Studentski trg 16, 11158 Belgrade,
Serbia}

\author{Jun Zhang}
\email{jun.zhang.3@umontreal.ca}
\address{Centre de Recherches Math\'ematiques, University of Montreal, C.P. 6128 Succ. Centre-Ville Montreal, QC H3C 3J7, Canada}

\begin{document}

\maketitle

\begin{abstract} We study the contact Floer homology ${\rm HF}_*(W, h)$ introduced by Merry-Uljarevi\'c in {\cite{Igor-Merry}}, which associates a Floer-type homology theory to a Liouville domain $W$ and a contact Hamiltonian $h$ on its boundary. The main results investigate the behavior of ${\rm HF}_*(W, h)$ under the perturbations of the input contact Hamiltonian $h$. In particular, we provide sufficient conditions that guarantee ${\rm HF}_*(W, h)$ to be invariant under the perturbations. This can be regarded as a contact geometry analogue of the continuation and bifurcation maps along the Hamiltonian perturbations of Hamiltonian Floer homology in symplectic geometry. As an application, we give an algebraic proof of a rigidity result concerning the positive loops of contactomorphisms for a wide class of contact manifolds. \end{abstract}


\section{Introduction}
Let $(M, \xi = \ker \alpha)$ denote a contact manifold with the co-oriented contact structure $\xi$ given by a contact 1-form $\alpha$. Through the entire paper, we will assume $M = \partial W$ where $W$ is a Liouville domain, an exact symplectic manifold endowed with the symplectic structure $\omega = d\lambda$ and $\alpha = \lambda|_M$. For any given smooth function $h: [0,1] \times M \to \R$, usually called a contact Hamiltonian, one obtains a dynamical system by solving the differential equation
\[\dot{\phi}_t= X_t\circ\phi_t,\]
where $X_t$ is the vector field of the contact Hamiltonian $h_t$, i.e. the vector field determined by
\begin{equation} \label{contact-ham}
    \alpha(X_t)=-h_t \,\,\,\,\mbox{and}\,\,\,\, d\alpha(X_t, \cdot )= dh_t - dh_t(R)\cdot \alpha.
\end{equation}
Here, $R$ denotes the Reeb vector field with respect to the contact form $\alpha$. The flow $\phi_t$ (or $\phi_t^{h}$ if we need to emphasize the contact Hamiltonian $h$) is called the contact Hamiltonian flow. 

\medskip

Mimicking the study of the Hamiltonian dynamics in symplectic geometry, one way to analyze the dynamical data from (\ref{contact-ham}) is to associate a Floer-type homology theory to $(M, \xi = \ker \alpha)$. Based on the generalized maximum principle in \cite{Igor-Merry}, one can construct such a homology theory, called the contact Floer homology of $h$ and denoted by ${\rm HF}_*(W, h)$ or simply ${\rm HF}_*(h)$ (if it is clear from the context what the filling $W$ is). If the flow $\phi_t$ has no 1-periodic orbits, ${\rm HF}_*(h)$ is well-defined as the Hamiltonian Floer homology ${\rm HF}_*(H)$ for a Hamiltonian $H:[0,1]\times\hat{W}\to \R$ on the completion $\hat{W}$ of the Liouville domain $W$. The Hamiltonian $H$ is assumed to satisfy
\begin{equation} \label{convex-end}
H_t(x,r)= r\cdot h_t(x)
\end{equation}
on the convex end of $\hat{W}$ where $\hat{W}$ admits the coordinates $(x,r)\in M\times[1,+\infty).$ The contact Hamiltonian $h$ is called the \emph{slope} of the Hamiltonian $H$. The well-definedness of ${\rm HF}_*(h)$ is deeply due to the Hamiltonian Floer homologies ${\rm HF}_*(H)$ being naturally isomorphic for different Hamiltonians $H$ that have {\bf the same slope}. Moreover, for brevity, the contact Hamiltonians with the corresponding contact Hamiltonian flows having no 1-periodic orbits are called {\it admissible}. The construction above generalizes the classical consideration of the Hamiltonians on $\hat{W}$ that have constant slopes not equal to a period of any Reeb orbit on $M=\partial W$. Those Hamiltonians are the building blocks of Viterbo's symplectic homology, ${\rm SH_*(W)}$,  of the Liouville domain $W$ (see \cite{Viterbo99}). 

\medskip

In the present paper, we investigate the behavior of the contact Floer homology ${\rm HF}_\ast(h)$ when the contact Hamiltonian $h$ is perturbed. 

\subsection{Isomorphism under perturbations}\label{sec:iso&pert}

Recall that in the closed symplectic manifold situation, once the Hamiltonian $H = H^0$ is perturbed to $H^1$, the corresponding Hamiltonian Floer homologies ${\rm HF}_*(H^0)$ and ${\rm HF}_*(H^1)$ are isomorphic via the so-called continuation map 
\begin{equation} \label{symp-per}
c: {\rm HF}_*(H^0) \to {\rm HF}_*(H^1)
\end{equation}
induced by a smooth family of functions $\{H^s\}_{s \in [0,1]}$ from $H^0$ to $H^1$. More explicitly, the map $c$ in (\ref{symp-per}), on the chain complex level, is defined by counting solutions of the $s$-dependent Floer equation that connect 1-periodic orbits of $H^0$ to 1-periodic orbits of $H^1$. For more details, see Section 6 in \cite{SZ92} as a standard reference. In fact, this provides an efficient way to prove the famous Arnold conjecture on the fixed points of Hamiltonian diffeomorphisms. As opposed to the closed case, for Hamiltonians on $\hat{W}$, the continuation map \eqref{symp-per} fails to be well-defined unless the slopes $h^0, h^1$ of $H^0, H^1$, respectively, satisfy the following condition:
\begin{equation} \label{monotone-cond}
h^0_t(x) \leq h^1_t(x) \,\,\mbox{for any $x \in M = \partial{W}$ and $t \in [0,1]$}.
\end{equation} 
For brevity, denote the condition (\ref{monotone-cond}) by $h^0 \leq h^1$. 
\medskip

As the main result in this paper, the following theorem provides a sufficient condition for the contact Floer homologies to be isomorphic when the contact Hamiltonian is perturbed. 

\begin{Th}\label{thm:bifiso}
Let $M$ be the boundary of a Liouville domain $W$. If $h^s:[0,1] \times M \to \R$, $s\in[0,1]$ is a smooth family of admissible contact Hamiltonians, then the contact Floer homologies ${\rm HF}_\ast(h^0)$ and ${\rm HF}_\ast(h^1)$ are isomorphic.
\end{Th}

A special case of Theorem~\ref{thm:bifiso}, where the contact Hamiltonains are assumed to be strict (i.e. the associated flows preserve not only the contact distribution but also the contact form) was proven in {\cite{drobnjak2021exotic}}. An easy case, that illustrates Theorem~{\ref{thm:bifiso}, is that of constant slopes, that is $h_t(x)=a$, where the value $a$ is not a period of a Reeb orbit on $M$.} Theorem~\ref{thm:bifiso} asserts that the contact Floer homology ${\rm HF}_\ast(a)$ does not change (up to isomorphism) if $a\in\R$ goes through numbers that are not periods of Reeb orbits on $M$. It is readily verified that, in this case, the promised isomorphism in Theorem \ref{thm:bifiso} can always be realized by the continuation map as in (\ref{symp-per}).

However, in contrast to the constant-slope case, the proof of the isomorphism between contact Floer homologies in general is not straightforward and is quite technically involved. Explicitly, for a given smooth family $h^s$ of admissible contact Hamiltonians on $[0,1] \times M$, one can construct a smooth family $H^s: [0,1] \times \hat{W}\to\R$ of non-degenerate Hamiltonians that all share the same set of 1-periodic orbits and such that $H^s$ has the slope equal to $h^s$. The desired isomorphism in Theorem \ref{thm:bifiso} then comes from the map between Hamiltonian Floer homologies 
\begin{equation} \label{bifur-map}
b: {\rm HF}_\ast(H^0)\to {\rm HF}_\ast(H^1) 
\end{equation}
that sends every 1-periodic orbit $x$ of $H^0$ to itself seen as a 1-periodic orbit of $H^1$. The map $b$ is obviously an isomorphisms, because no additional 1-periodic orbits are created and no existing ones are lost along the perturbation $H^s$. In other words, the map $b$ is a trivial bifurcation map. The proof of Theorem \ref{thm:bifiso} is given by Section \ref{sec-main-proof} with more details provided in Section \ref{sec-lemmas}. 

\subsection{Comparisons of continuation and bifurcation}

In Section~\ref{sec:iso&pert} above, the Floer chain complexes ${\rm CF}_\ast(H^0)$ and ${\rm CF}_\ast(H^1)$ are identical. Namely, all the objects used in the construction of ${\rm CF}_\ast(H^0)$ and ${\rm CF}_\ast(H^1)$ (the 1-periodic orbits and the Floer cylinders) are contained in the region of $\hat{W}$ where $H^0$ and $H^1$ coincide. The map $b$ is equal to the identity already on the level of chain complexes. It is, therefore, a natural question whether this map $b$ complements the continuation maps, i.e. whether $b$ coincides with the continuation map whenever the slope of $H^0$ is less than or equal to the slope of $H^1$. Somewhat surprisingly, the answer is negative. This is a consequence of the next theorem, which asserts that there are situations in which both the isomorphism $b$ and the correcponding continuation map are well-defined but in which the continuation map is not an isomorphism.

Before we state the result, let us recall some notions. Denote by ${\rm Cont}(M, \xi)$ the group of contactomorphisms. A loop in ${\rm Cont}(M, \xi)$ is called positive if it can be generated by a contact Hamiltonian which is pointwise positive. A contact manifold $M$ is called non-orderable if there exists a contractible positive loop in ${\rm Cont}(M, \xi)$ (see Proposition 2.1.A and Proposition 2.1.B in \cite{EP00}). 

\begin{Th}\label{thm:bifvscont}
Let $W^{2n}$ be a Liouville domain whose contact boundary $M = \partial W$ is non-orderable. Assume that the symplectic homology ${\rm SH}_\ast(W)$ is not isomorphic to the (shifted) singular homology ${\rm H}_{\ast+n}(W,\partial W).$ Then, there exists a smooth $s$-family $h^s: [0,1] \times \partial W\to\R$ of admissible contact Hamiltonians where $h^0\leq h^1$ such that the continuation map $c: {\rm HF}_\ast(h^0)\to {\rm HF}_\ast(h^1)$ is {\rm not} an isomorphism.
\end{Th}

In fact, the condition of  ${\rm SH}_\ast(W)$ and ${\rm H}_{\ast +n}(W,\partial W)$ not being isomorphic can be replaced by the following formally weaker condition: there exists a positive time-independent admissible contact Hamiltonian $g:M\to\R$ such that the canonical map ${\rm HF}_*(g)\to {\rm SH}_*(W)$ is not an isomorphism. Conjecturally, this condition always holds. On the other hand, the next result shows that the continuation map and the bifurcation isomorphism do coincide if we add a few additional assumptions. This will be illustrated by the proof of the following theorem. 

\begin{Th}\label{thm:contisiso}
Let $M$ be the contact boundary of a Liouville domain, and let $h^s: [0,1] \times M \to \R$ be a smooth $s$-family of admissible contact Hamiltonians. Assume that $h^0$ is time-independent and positive on $M$, and $h^s$ satisfies the following condition
\[ h^0 \leq h^s\,\,\,\,\mbox{for all $s\in[0,1]$}.\]
Then, the continuation map ${\rm HF}_\ast(h^0)\to {\rm HF}_\ast(h^1)$ is an isomorphism.
\end{Th}

The proofs of Theorem \ref{thm:bifvscont} and Theorem \ref{thm:contisiso} are given in Section \ref{sec-main-proof}. 

\subsection{Rigidity of positive loops} Theorem \ref{thm:contisiso} has a direct application to the study of positive loops in ${\rm Cont}(M, \xi)$. The notion of positive loops in ${\rm Cont}(M, \xi)$ has been introduced by Eliashberg and Polterovich in \cite{EP00} in the context of orderability of contact manifolds. Ever since, positive loops have played an important role in the study of contact geometry. A remarkable result along these lines is the relation between contractible positive loops in ${\rm Cont}(M, \xi)$ and the contact non-squeezing phenomena (especially on contact balls $B^{2n}(R) \times S^1$) discovered in \cite{EKP06} (see Section~1.7, Theorem~1.3, and Theorem~1.5). More results in this direction can be found in {\cite{San11}},  Theorem 1.5 in \cite{Chiu17}, Theorem 1.2 in \cite{{Fra16}}, and Theorem 1.22 in \cite{Alber-Merry}.

It has been observed that for some contact manifolds, for which the contact non-squeezing theorem holds, positive loops exhibit a certain rigidity, as a matter of fact, the contracting homotopies of positive loops do. More precisely, as shown in Theorem 1.11 in \cite{EKP06} or more generally in Theorem 1.25 in \cite{Alber-Merry}, contact non-squeezing implies that the $s$-family of contact Hamiltonians $h_t^s$ furnished by the contracting homotopy must have a sufficiently small negative value (large in absolute value). In particular, a positive loop is not contractible through positive loops. On the other hand, a sophisticated result called the contact systolic inequality (see Theorem 1.2 in \cite{Alber-Fuchs-Merry}) shows that for {\it any} positive loop $\phi$ in ${\rm Cont}(M, \xi = \ker \alpha)$, there exists a positive number $C(\alpha, [\phi])$ such that 
\[ \|h_t\| \geq C(\alpha, [\phi]) \,\,\,\,\mbox{for some norm $\|\cdot\|$ depending on $\alpha$}, \]
where $[\phi]$ denotes the homotopy class of the loop $\phi$. In particular, a positive loop cannot be contractible through positive loops. This holds for any contact manifold (even if it is not fillable), so it generalizes the main result in \cite{CPS16} which only concentrates on overtwisted contact manifolds. 

Both approaches above imply that a positive loop in ${\rm Cont}(M, \xi)$ cannot be contractible through positive loops only. Here, as an application of Theorem \ref{thm:contisiso}, we provide a direct {\it algebraic} proof of this result in the following special case. The novelty is that our proof bypasses both the contact non-squeezing theorem and any quantitative study of the length of positive loops.

\begin{Th} \label{thm-pos-contr}
Let $W$ be a Liouville domain. Assume there exists a time-independent admissible positive contact Hamiltonian $g:\partial W \to \R$ such that the canonical map
\begin{equation} \label{can-map}
{\rm HF}_\ast(g)\to {\rm SH}_\ast(W)
\end{equation}
is {\rm not} an isomorphism. Then, there are no positive loops of contactomorphisms on $\partial W$ that are contractible through positive loops. 
\end{Th}

\begin{Rem} Here are three remarks on the hypothesis of Theorem \ref{thm-pos-contr}. (i) For any contact manifold $(M, \xi)$, there always exists admissible contact Hamiltonian $g$ satisfying the condition in Theorem \ref{thm-pos-contr}. Indeed, consider $g \equiv \ep$ for a sufficiently small constant $\ep>0$. Then, due to Yorke's result in \cite{Yorke-1969-PAMS}, its time-1 map does not have any fixed points. Therefore, the contact Floer homology ${\rm HF}_*(g)$ is well-defined. (ii) The symplectic homology ${\rm SH}_*(W)$ is equal to a direct limit of ${\rm HF}_*(g)$ for an increasing family of admissible contact Hamiltonians. Therefore, we have the canonical map ${\rm HF}_*(g) \to {\rm SH}_*(W)$ that appears in (\ref{can-map}). (iii) The condition (\ref{can-map}) can be verified in a rather trivial way in many cases. In subsection \ref{sec-ex}, we provide a list of examples that satisfy (\ref{can-map}). \end{Rem}

\begin{Rem} Everywhere in the main results in this paper, the hypothesis is stated that the contact manifold $M$ is the boundary of a Liouville domain $W$, that is, $M$ is Liouville fillable. In fact, the proofs of the results above are valid also for $M$ that is strongly fillable by a symplectic manifold $W$ on which the symplectic homology theory is well defined. For instance, they are valid if the filling $W$ is ${\mbox{weakly}}^+$ monotone {\cite{HS95}}. \end{Rem}

\noindent{\bf Acknowledgement.} This work was completed while the second author held a CRM-ISM Postdoctoral Research Fellowship at the Centre de recherches math\'ematiques in Montr\'eal. He thanks this Institute for its warm hospitality. This paper is partly motivated by discussions with Egor Shelukhin on the positive loops of contactomorphisms, so we thank for his inspiration. This research was partially supported by the Science Fund of the Republic of Serbia, grant no.~7749891, GWORDS.

\section{Proofs of Theorems \ref{thm:bifiso}, \ref{thm:bifvscont}, and \ref{thm:contisiso}} \label{sec-main-proof}

In this section, we prove the main results, Theorems \ref{thm:bifiso}, {\ref{thm:bifvscont}}, and {\ref{thm:contisiso}, using technical lemmas from Section {\ref{sec-lemmas}}.}

\subsection{Proof of Theorems \ref{thm:bifiso} and \ref{thm:contisiso}} The idea of the proof is to construct a smooth $s$-family $\left(H^s, J^s\right)$ for $s\in[0,1]$ of regular Floer data on the completion $\hat{W}$ such that $H^0$ has slope $h^0$ and such that $H^1$ has slope $h^1$. In our construction,  $\left(H^s, J^s\right)$ for different choices of $s\in[0,1]$ differ only on a conical end where the corresponding Hamiltonians do not have any 1-periodic orbits. 

Let $(H^0, J^0)$ be regular Floer data for the contact Hamiltonian $h^0$. By definition, there exists a positive real number $r_H\in\R^+$ such that $H^0_t(x,r)= r\cdot h^0_t(x)$ for $r\geq r_H$. Let $C\in\R^+$ be such that
\[\ln C\geq \max_{s,t\in[0,1], \,\,x\in M} \abs{dh^s_t(x)(R(x))}.\]
By Lemma~\ref{lem:existsmu}, there exists a smooth function $\mu: \R^+\to [0,1]$ with the following properties:
\begin{enumerate}
    \item the Hamiltonian $G^s_t : M \times\R^+ \to \R$ defined by $(x,r)\mapsto r\cdot h^{s\cdot \mu(r)}_t(x)$ has no 1-periodic orbits for all $s\in[0,1]$;
    \item $\mu(r)=0$ for $r\in\left(0, C^2\cdot r_H\right]$ and $\mu(r)=1$ for $r$ large enough.
\end{enumerate}
For brevity, denote $W^{r} : = \hat{W} \backslash \left(M \times (r, \infty)\right)$ for any $r \geq 0$. Let $(H^s, J^s)$ be regular Floer data on $\hat{W}$ for every $s\in[0,1]$ such that the following conditions hold:
\begin{itemize}
\item[(1)] $H^s_t(x,r)= G^s_t(x,r)$ for $(x,r)\in M \times [r_H, \infty)$;
\item[(2)] $H^s_t(p)= H^0_t(p)$ for $p\in \op{int} W^{r_H}$ and $J^s= J^0$ on $\op{int} W^{C\cdot r_H}$. 
\end{itemize}
By the construction, $H^s$ coincides with $G^s$ on $M \times[r_H, \infty)$ and with $H^0$ on $\op{int} W^{C^2\cdot r_H}.$ Then Lemma~\ref{lem:investimate} implies that the set $\left\{ \phi_t^{H^s}(p)\:|\: t\in[0,1] \right\}$ does not intersect $M \times\left[ C^2\cdot r_H, \infty \right)$ for any $s\in[0,1]$ if $p\in \op{int} W^{C\cdot r_H}.$ Therefore, we have
\[\phi_t^{H^s}(p)= \phi_t^{H^0}(p)\]
for all $s,t\in[0,1]$ and $p\in \op{int} W^{C\cdot r_H}$. Similarly, the set $\left\{ \phi_t^{H^s}(p)\:|\: t\in[0,1] \right\}$ does not intersect the set $\op{int} W^{r_H}$ for any $s\in[0,1]$ if $p\in M \times[C\cdot r_H, \infty)$. Hence, we have
\[\phi_t^{H^s}(p)= \phi_t^{G^s}(p) \]
for all $s, t\in[0,1]$ and $p\in M \times[C\cdot r_H, \infty)$.

Finally, since $M\times[C\cdot r_H, \infty)$ and $\op{int} W^{C\cdot r_H}$ cover $\hat{W}$, the Hamiltonian $H^s$ has no 1-periodic orbits apart from the ones in $\op{int} W^{C\cdot r_H}$. Moreover, on $\op{int} W^{C\cdot r_H}$, the Hamiltonians $H^s$ coincide for different $s\in[0,1]$. Therefore, the desired isomorphism $b: {\rm HF}_\ast(h^0)\to {\rm HF}_\ast(h^1)$ is obtained by mapping each generator of ${\rm HF}_\ast(H^0, J^0)$ to itself seen as a generator of ${\rm HF}_\ast(H^1, J^1)$. This isomorphism is occasionally called \emph{bifurcation isomorphism}.

\medskip

Now, we prove that the bifurcation isomorphism coincides with the continuation map if $h^0$ is autonomous and if $h^s_t(x)\geqslant h^0(x)> 0$ for all $s\in[0,1]$ and $x\in M.$ In this case, $H^0\leqslant H^1 $, pointwise. Hence, there exists continuation data  $\{(G^s, I^s)\}$ from $(H^0, J^0)$ to $(H^1, J^1)$. Since $H^0=H^1$ on $\op{int} W^{C \cdot r_H}$, the Hamiltonian $G^s$ coincides with $H^0$ there for all $s$. One can choose $J^s$ such that $J^s=J^0$ on this set. Let $V$ be the complement of the set $\{(x,r)\in M\times\R^+\:|\: h^0(x)\cdot r>1\}$ in $\hat{W}$. By decreasing $r_H$ and increasing $C$ if necessary, one may assume 
\[\partial V \subset \op{int} W^{C \cdot r_H}\setminus W^{r_H}.\]
We modify $J^0$, if necessary, so that it is of contact type along $\partial V$. By the no-escape lemma, Lemma 19.3 in {\cite{Ritter-TQFT}}, the solutions of the $s$-dependent Floer equation
\[\partial_s u + J^s_t(u)\left( \partial_t u - X^{H^s_t}(u) \right) = 0\]
with finite energy are entirely contained in $\op{int} W^{C \cdot r_H}$. In this region, the continuation data $\{(H^s, J^s)\}$ is $s$-independent. Therefore, by the standard argument in Floer theory, there are no Floer cylinders that connect different $1$-periodic orbits. As a consequence, the continuation map $c: {\rm HF}_\ast(H^0, J^0)\to {\rm HF}_\ast(H^1, J^1)$ coincides with the bifurcation isomorphism $b: {\rm HF}_\ast(h^0)\to {\rm HF}_\ast(h^1)$. In particular, the continuation map $c$ is an isomorphism. \qed

\subsection{Proof of Theorem \ref{thm:bifvscont}}  Assume the contrary, i.e. that the continuation map
\[c: {\rm HF}_*(h^0)\to {\rm HF}_*(h^1)\]
is an isomorphism whenever the contact Hamiltonians $h^0\leqslant h^1$ can be joined by a path of admissible contact Hamiltonians. Since $M$ is non-orderable, there exists a contractible positive loop $\phi_t:M\to M$ of contactomorphisms. Denote by $\phi^s_t$, $s\in[0,1]$ the contracting homotopy, i.e. $\phi_t^0={\rm id}$ for all $t$, $t\mapsto \phi_t^s$ is a loop of contactomorphisms for all $s$, and $\phi_t^1=\phi_t$ for all $t$. Let $h^s_t:M\to\R$ be the contact Hamiltonian of $\phi_t^s$ that is given by (\ref{contact-ham}). Since $\phi_t$ is a positive loop, $h^1>0$. Let $\varepsilon>0$ be a sufficiently small number such that no Reeb orbit on $M$ has a (positive) period less than or equal to $\varepsilon$. Denote by $\phi_t^R:M\to M$ the Reeb flow on $M$, and let $f_t^s:M\to \R$ be the contact Hamiltonian defined by
\begin{equation} \label{cocycle}
f_t^s:= \varepsilon + h_t^s\circ\left( \phi_{\varepsilon t}^R \right)^{-1}.
\end{equation}
By a contact version of the cocycle formula (see Lemma 2.2 in \cite{Muller-Speath}), the flow of the contact Hamiltonian $f_t^s$ is equal to $\phi_{\ep t}^R\circ \phi_t^s$. Since $t\mapsto \phi_t^s$ is a loop of contactomorphisms for all $s$, $\{f_t^s\}_{s\in[0,1]}$ is a smooth family of admissible contact Hamiltonians that join $f^0=\varepsilon$ and $f^1\geqslant\varepsilon$. Therefore, by the assumption, the continuation map
\[{\rm HF}_\ast(\ep)\to {\rm HF}_\ast(f^1)\]
is an isomorphism. Since $\phi_t$ is a contractible positive loop, then so is its $k$-th iterate $t\to \phi_{kt}$. The contact Hamiltonian of the $k$-th iterate is equal to $h^{(k)}_t:=k\cdot h^1_{kt}$. The argument above implies that the continuation map
\[{\rm HF}_*(\ep)\to {\rm HF}_*(f^{(k)})\]
is an isomorphism, where $f_t^{(k)}:M\to R$ is the (positive) contact Hamiltonian given by
\begin{equation} \label{k-iterate}
f_t^{(k)} = \varepsilon + h_t^{(k)}\circ \left( \phi_{\varepsilon t}^R\right)^{-1}. 
\end{equation}
Importantly, observe that we can find a sequence $\{f_t^{(k_i)}\}_{i \in \N}$ where $k_1 = 1$ such that for every $t \in [0,1]$, we have 
\[ f_t  < f_t^{(k_2)} < f_t^{(k_3)} < \cdots. \]
Indeed, for instance, in order to obtain $k_2$, we need to compare $h_t$ with $kh_{kt}$. Since our manifold $M$ is compact and the time interval $[0,1]$ is also compact, there exists a global maximum of $h_t$ on $M$ over $t \in [0,1]$. Then take $k_2$ sufficiently large so that $\min_{t \in [0,1]} k_2h_{k_2 t}$ on $M$ is larger than this global maximum. Inductively, we obtain all the desired $k_i$.  By our assumption, the continuation map ${\rm HF}_*(f^{(k_i)}) \to {\rm HF}_*(f^{(k_{i+1})})$ is an isomorphism for all $i \in \N$. Therefore, we have successive isomorphisms, 
\begin{equation} \label{suc-iso}
{\rm HF}_*(\ep)  \simeq {\rm HF}_*(f^{(k_1)}) \simeq {\rm HF}_*(f^{(k_2)}) \simeq {\rm HF}_*(f^{(k_3)}) \simeq \cdots,
\end{equation}
Moreover, we have $\displaystyle \lim_{\longrightarrow} {\rm HF}_*(f^{(k_i)})= {\rm SH}_\ast(W)$ (see Section 4 in \cite{Igor-Merry}). This implies that the canonical map
\[{\rm HF}_*(\varepsilon)\to {\rm SH}_*(W)\]
is an isomorphism. By the last equality in \cite{Igor-Merry}, the contact Floer homology ${\rm HF}_*(\varepsilon)$ is isomorphic to ${\rm H}_{*+n}(W, \partial W)$. Therefore, we obtain the desired contradiction and this finishes the proof.\qed

\section{Proof of Theorem \ref{thm-pos-contr} and examples} \label{sec-proof}

The proof of Theorem \ref{thm-pos-contr} is rather similar to the proof of Theorem \ref{thm:bifvscont}. 

\subsection{Proof of Theorem \ref{thm-pos-contr}}  Assume there exists a positive loop of contactomorphisms $\phi_t: \partial W\to \partial W$ for $t \in [0,1]$ that is contractible through non-negative loops of contactomorphisms. In other words, there exists an $s$-family of contact Hamiltonians $h^s: [0,1] \times \partial W\to [0,+\infty)$ such that $h^s$ generates a loop of contactomorphisms $\phi_t^{s}:\partial W\to \partial W$ for all $s\in[0,1]$ and such that $h:=h^1$ is the contact Hamiltonian of the loop $\{\phi_t\}_{t\in[0,1]}$. Denote by $\{\phi^g_t\}_{t \in [0,1]}$ the contact Hamiltonian flow generated by $g$. Similarly to (\ref{cocycle}), the composition $\{\phi_{t}^g \circ \phi_t^s\}_{t \in [0,1]}$ is generated by the contact Hamiltonian 
\begin{equation} \label{ch-comp}
f_t^s := g + h_t^s \circ (\phi^g_{t})^{-1}. 
\end{equation}
For any point $x \in \partial W$, since $g$ is admissible and since $\phi_1^s = \mathds{1}$, the identity map on $\partial W$, by definition we have 
\begin{equation} \label{noreeb-comp}
x \neq \phi^g_{1}(x) = (\phi_{1}^g \circ \phi_1^s)(x).
\end{equation}
In other words, the composition flow $\{\phi_{t}^g \circ \phi^s_t\}_{t \in [0,1]}$ has no $1$-periodic orbits for all $s\in[0,1]$. Hence, the contact Hamiltonian $f^s$ is admissible for all $s\in [0,1]$. Different from the proof of Theorem \ref{thm:bifvscont}, here since $h_t^s$ is positive, we have $0< f^0= g \leq f^s_t$ pointwise for all $s \in [0,1]$. Therefore, Theorem \ref{thm:contisiso} applies and yields an isomorphism
\[{\rm HF}_\ast(g)= {\rm HF}_*(f^0)\to {\rm HF}_*(f^1)= {\rm HF}_\ast(f),\]
importantly, induced by the continuation map. Similarly to the proof of Theorem \ref{thm:bifvscont}, consider the $k$-th iterate of $\phi$ which is generated by the contact Hamiltonian $h^{(k)}_t:= k\cdot h_{kt}$. By the same argument as above, we can find a sequence $\{f_t^{(k_i)}\}_{i \in \N}$ where $k_1 = 1$ such that for every $t \in [0,1]$, we have 
\[ f_t  < f_t^{(k_2)} < f_t^{(k_3)} < \cdots. \]
This implies successive isomorphisms, 
\begin{equation} \label{suc-iso}
{\rm HF}_*(g)  = {\rm HF}_*(f^{(k_1)}) \simeq {\rm HF}_*(f^{(k_2)}) \simeq {\rm HF}_*(f^{(k_3)}) \simeq \cdots,
\end{equation}
where all the isomorphisms are induced by the continuation maps. The rest of the proof goes in the same manner as the end of the proof of Theorem \ref{thm:bifvscont}. The isomorphism ${\rm HF}_*(g) \simeq {\rm SH}_*(W)$ provides the desired contradiction. \qed

\subsection{Examples} \label{sec-ex} 
In this section, we list some examples in which one can easily verify that ${\rm HF}_\ast(g)$ and ${\rm SH}_\ast(W)$ from Theorem~\ref{thm-pos-contr} are not isomorphic. Recall that if $\ep>0$ is sufficiently small, the groups ${\rm HF}_*(\ep)$ and ${\rm H}_{*+n}(W, \partial W)$ are isomorphic. Therefore, the non-equality (up to isomorphism) of ${\rm H}_{*+n}(W, \partial W)$ and ${\rm SH}_*(W)$ initiates the application of Theorem  {\ref{thm-pos-contr}}. Interestingly, comparing ${\rm H}_{*}(W, \partial W)$ and ${\rm SH}_*(W)$ is also fundamentally important when one verifies the Weinstein conjecture for $\partial W$ in \cite{Viterbo99}. Our examples come from the two extreme situations where either ${\rm SH}_\ast(W)=0$ or ${\rm SH}_\ast(W)$ is infinite dimensional. In the latter case, even a stronger conclusion than that of Theorem~\ref{thm-pos-contr} holds: there are no contractible positive loops of contactomorphisms on $\partial W$ if ${\rm SH}_\ast(W)$ is infinite dimensional (see Theorem~1.5 in \cite{Igor-Merry}). 

\begin{Ex} \label{ex-1} Suppose $W$ is a compact symplectic manifold of dimension $2n$ with boundary. Recall that $W$ is called $\bf k$-orientable if $W - \partial W$ is a $\bf k$-oriented manifold without boundary. Then Lemma 3.27 in \cite{Hatcher-book} shows that ${\rm H} _{2n}(W; \partial W; \bf k)$ admits a fundamental class. Hence, in particular, ${\rm H}_*(W; \partial W; {\bf k}) \neq 0$. 

On the other hand, the vanishing of the entire ${\rm SH}_*(W) = {\rm SH}_*(W; \bf k)$ appears quite often in symplectic geometry and contact geometry. Here we give several examples. By \cite{Cieliebak-handle-JEMS,Fau-20} and the basic computation of ${\rm SH}_*(B^{2n}(R); \bf k)$, any subcritical Stein domain $W$ has its ${\rm SH}_*(W;\bf k) =0$. By Corollary 6.5 in \cite{Seidel-bias-SH}, any Liouville filling $W$ of $(S^{2n-1}, \xi_{\rm std})$ has ${\rm SH}_*(W; {\bf k}) = 0$. By Theorem A.1 in \cite{Kang-displace} or Theorem 13.4 in \cite{Ritter-TQFT}, if a Liouville domain $W$ is displaceable in its symplectic completion, then ${\rm SH}_*(W; {\bf k})=0$. In fact in  \cite{Ritter-TQFT}, several algebraic criterions are provided to guarantee the vanishing of ${\rm SH}_*(W; \bf k)$, see its Section 10 and Theorem 13.3. For these cases, we have $0 = {\rm SH}_*(W; {\bf k}) \neq {\rm H}_*(W, \partial W; {\bf k}) \neq 0$.\end{Ex}

\begin{Ex} \label{ex-2} For $W = D^*Q$, the unit codisk bundle of a closed manifold $Q$, a well-known result by Viterbo \cite{Viterbo99,Abbondandolo-Schwarz-06,AS14,Abouzaid-symplectic-book,Salamon-Weber}, says that ${\rm SH}_*(D^*Q; {\bf k}) \simeq {\rm H}_*(\Lambda Q; {\bf k})$ when $Q$ satisfies some topological condition (otherwise one can use twisted coefficients), where $\Lambda Q$ is the free loop space. Then, by a theorem in \cite{Vigue-Poirrier-Sullivan} (see also Remark 1.1 in \cite{Alber-Frauenfelder-LWI-CB}), ${\rm dim}_{\bf k} {\rm H}_*(\Lambda Q; {\bf k}) = \infty$ if $\pi_1(Q)$ is finite. In particular, in this case, ${\rm SH}_*(W; {\bf k}) \neq {\rm H}_*(W, \partial W; {\bf k})$ since ${\rm dim}_{\bf k} {\rm H}_*(W, \partial W; {\bf k}) = {\rm dim}_{\bf k}{\rm H}_*(D^*Q, S^*Q; {\bf k})$ is always bounded. \end{Ex}

\begin{Ex} \label{ex-3} For a contact manifold $M$ obtained from smoothing the codimension two corners of the boundary of a certain Lefschetz fibration (see the beginning of \cite{Albers-McLean-LWI}), any strong symplectic filling $W$ of $M$ such that the pull-back $H^1(W)\to H^1(M)$ of the inclusion $M\hookrightarrow W$ is surjective has ${\rm SH}_*(W; {\bf k})$ infinite dimensional (see Theorem 1.2 in {\cite{Albers-McLean-LWI}}). By the same reason as in Example \ref{ex-2}, we have ${\rm SH}_*(W; {\bf k}) \neq {\rm H}_*(W, \partial W; {\bf k})$.\end{Ex}

\section{Technical lemmas} \label{sec-lemmas}

In this section, we provide detailed proofs of the lemmas that appeared in the proof of Theorem \ref{thm:bifiso}. The next lemma was used in the proof of Theorem~\ref{thm:bifiso} to show that changing a Hamiltonian on $\hat{W}$ far into a conical end does not affect its flow in a  chosen compact subset.

\begin{Lemma}\label{lem:investimate}
Let $\Sigma$ be a smooth manifold, let $X_t$, $t\in\R$ be a time-dependent smooth vector field on $\Sigma\times\R^+$ and let $\gamma:[0,a]\to\Sigma\times\R^+$ be an integral curve of $X$ such that $\gamma(0)\in\Sigma\times\{r_0\}$ and $\gamma(a)\in\Sigma\times\{r_1\}$. Denote by $\pi:\Sigma\times\R^+\to\R^+$ the projection and by $g$ the Riemannian metric on $\R^+$ given by $g:=\frac{dr\otimes dr}{r^2}$. Then,
\[r_0\cdot e^{-a\cdot C}\leq r_1\leq r_0\cdot e^{a\cdot C},\]
where $C:= \sup_{p\in\Sigma\times\R^+, \,t\in \R}\norm{d\pi(X_t(p))}_g.$
\end{Lemma}
\begin{proof}
The distance between $r_0$ and $r_1$ in $\left( \R^+, g \right)$ is equal to $\abs{\ln r_1 - \ln r_0}$. Since $\pi\circ\gamma$ is a smooth curve that joins $r_0$ and $r_1$,
\begin{align*}
    \abs{\ln \frac{r_1}{r_0}} \leq \op{length}(\pi\circ\gamma)& = \int_0^a \norm{\frac{d}{dt}(\pi\circ\gamma(t))}_g dt \\
    & = \int_0^a \norm{d\pi (\gamma'(t))}_g dt = \int_0^a \norm{d\pi (X_t(\gamma(t)))}_g dt \leq a\cdot C.
\end{align*}
Therefore, we have $r_0\cdot e^{-a\cdot C}\leq r_1\leq r_0\cdot e^{a\cdot C}$, as desired.
\end{proof}

The following three lemmas (Lemma~\ref{lem:gronwall}, Lemma~\ref{lem:gronapp}, and Lemma~\ref{lem:existsmu}) prove that one can interpolate between $r\cdot h^0$ and $r\cdot h^1$ on a conical end (in the situation of Theorem~\ref{thm:contisiso}) without creating 1-periodic orbits. This fact is crucial in the proof of Theorem~\ref{thm:contisiso}. Lemma~\ref{lem:gronwall} proves a Gronwall-type estimate that is used later in Lemma~\ref{lem:gronapp} and, indirectly, in Lemma~\ref{lem:existsmu}. 

\begin{Lemma}\label{lem:gronwall}
Let $(M,g)$ be a Riemannian manifold and let $\nabla$ be its Levi-Civita connection. Let $X_t, Y_t, t\in\R$ be two smooth time-dependent vector fields. Let $u:[0,1]\times[0, t_0]\to M$ be a smooth map with the following property: $t\mapsto u(s,t)$ is an integral curve of the time dependent vector field $(1-s)\cdot X_t + s\cdot Y_t$ for all $s\in[0,1]$, i.e.
\[\partial_t u(s,t) = (1-s)\cdot X_t(u(s,t)) + s\cdot Y_t(u(s,t))\]
for all $(s,t)\in[0,1]\times[0,t_0]$. Denote by $d$ the metric on $M$ induced by $g$, and denote
\begin{align*}
    &\norm{X}_u:=\underset{t\in \R}{\sup_{p\in\:\op{im} u}}\norm{X_t(p)}_g:=\underset{t\in \R}{\sup_{p\in\:\op{im} u}}\sqrt{g(X_t(p), X_t(p))},\\
    &\norm{\nabla X}_u:=\underset{t\in \R}{\sup_{p\in\:\op{im} u}}\norm{\nabla X_t(p)}_g :=\underset{t\in \R}{\sup_{p\in\:\op{im} u}}\sup_{v\not=0}\frac{\norm{(\nabla_v X_t)(p)}_g}{\norm{v}_g}.
\end{align*}
Then,
\[d\big(u(0,t), u(1,t)\big)\leq \alpha(t)\cdot e^{\beta(t)}\]
for all $t\in[0,t_0]$, where
\begin{align*}
    \alpha(t)&= \sqrt{\int_0^1 \norm{\partial_s u(s,0)}_g^2ds + t\cdot \norm{Y-X}^2_u},\\
    \beta(t)&= t\cdot \left( \max\{ \norm{\nabla X}_u, \norm{\nabla Y}_u \}+\frac{1}{2}\right).
\end{align*}
\end{Lemma}
\begin{proof}
We adapt the proofs of Proposition~1.1 and Theorem~1.2 in \cite{kunzinger2006global}. Denote by $\ell(t)$ and $e(t)$ the length and the energy of the curve
\[[0,1]\to M\quad:\quad s\mapsto u(s,t),\]
respectively. In other words,
\begin{equation*}
    \ell(t):= \int_0^1\norm{\partial_s u(s,t)}_g ds \,\,\,\,\mbox{and}\,\,\,\,\,e(t):=\int_0^1\norm{\partial_s u(s,t)}_g^2 ds.
\end{equation*}
By the Cauchy-Schwarz inequality, $\ell(t)\leq\sqrt{e(t)}$. The Newton-Leibniz formula, together with $\nabla_{\partial_t}\partial_s u= \nabla_{\partial_s}\partial_t u,$ implies
\begin{align*}
    e(a)- e(0) &= \int_0^a \partial_t e(t) dt\\
    &= \int_0^a \frac{d}{dt} \int_0^1\norm{\partial_s u(s,t)}_g^2 ds dt\\
    &= \int_0^a \int_0^1 \frac{d}{dt} g\left(\partial_s u(s,t), \partial_s u(s,t)\right) ds dt\\
    &= \int_0^a \int_0^1 2\cdot g\left(\nabla_{\partial_t}\partial_s u(s,t), \partial_s u(s,t)\right) ds dt\\
    &= 2\cdot\int_0^a \int_0^1 g\left(\nabla_{\partial_s}\partial_t u(s,t), \partial_s u(s,t)\right) ds dt.
\end{align*}
Denote $Z_t^s:=(1-s)\cdot X_t + s\cdot Y_t$. Since
\begin{align*}
    \nabla_{\partial_s}\partial_t u(s,t) &= \left( \partial_s Z^s_t \right)\circ u(s,t) + \left( \nabla_{\partial_s u(s,t)} Z_t^s \right)\circ u(s,t)\\
    &= (Y_t-X_t)\circ u(s,t) + \left( \nabla_{\partial_s u(s,t)} Z_t^s \right)\circ u(s,t),
\end{align*}
the triangle inequality implies
\[e(a)-e(0)\leq A+ B,\]
where
\begin{align*}
    & A:= 2\cdot\abs{\int_0^a \int_0^1 g\left((Y_t-X_t)\circ u(s,t), \partial_s u(s,t)\right) ds dt},\\
    & B:= 2\cdot\abs{\int_0^a \int_0^1 g\left(\left( \nabla_{\partial_s u(s,t)} Z_t^s \right)\circ u(s,t), \partial_s u(s,t)\right) ds dt}.
\end{align*}
The Cauchy-Schwarz inequality implies
\begin{align*}
    A &\leq 2\cdot\int_0^a\sqrt{\int_0^1 \norm{\left(Y_t-X_t\right)\circ u(s,t)}_g^2 ds\cdot \int_0^1 \norm{\partial_s u(s,t)}_g^2 ds  } dt\\
    &\leq \int_0^a 2\cdot\norm{Y-X}_u\cdot \sqrt{e(t)}dt.
\end{align*}
By the inequality between arithmetic and geometric means,
\[ 2\cdot\norm{Y-X}_u\cdot \sqrt{e(t)}\leq \norm{Y-X}_u^2 + e(t).  \]
Hence,
\[ A\leq \int_0^a \left( \norm{Y-X}_u^2 + e(t) \right) dt = a\cdot \norm{Y-X}_u^2 + \int_0^a e(t) dt. \]
Now, we apply the Cauchy-Schwarz inequality to the term $B$
\[ B\leq 2\cdot\int_0^a\sqrt{\int_0^1\norm{\left(\nabla_{\partial_s u(s,t)}Z^s_t\right)\circ u(s,t)}_g^2ds\cdot \int_0^1 \norm{\partial_s u(s,t)}_g^2ds} dt. \]
Since
\begin{align*}
    \norm{\left(\nabla_{\partial_s u(s,t)}Z^s_t\right)\circ u(s,t)}_g &\leq \norm{\left(\nabla Z^s_t\right)\circ u(s,t)}_g \cdot \norm{\partial_s u(s,t)}_g\\
    &\leq \norm{\nabla Z^s}_u\cdot \norm{\partial_s u(s,t)}_g\\
    &= \norm{ (1-s)\cdot \nabla X + s\cdot \nabla Y }_u\cdot \norm{\partial_s u(s,t)}_g\\
    &\leq \max\{ \norm{\nabla X}_u, \norm{\nabla Y}_u \}\cdot\norm{\partial_s u(s,t)}_g,
\end{align*}
the following inequality holds
\begin{align*}
    B&\leq 2\cdot \max\{ \norm{\nabla X}_u, \norm{\nabla Y}_u \}\cdot \int_0^a\int_0^1\norm{\partial_s u(s,t)}^2_g ds dt\\
    &= 2\cdot \max\{ \norm{\nabla X}_u, \norm{\nabla Y}_u \}\cdot \int_0^a e(t) dt.
\end{align*}
Consequently,
\begin{align*}
    e(a)-e(0)&\leq A+B\\
    &\leq a\norm{Y-X}^2_u + \int_0^a e(t) dt + 2\max\{ \norm{\nabla X}_u, \norm{\nabla Y}_u \}\cdot \int_0^a e(t) dt\\
    &\leq a\norm{Y-X}^2_u + \bigg(2\max\{ \norm{\nabla X}_u, \norm{\nabla Y}_u \}+1\bigg)\cdot \int_0^a e(t) dt.
\end{align*}
The Gr\"onwall inequality implies
\[ e(t)\leq \left( e(0) + t\norm{Y-X}^2_u \right)\cdot \exp \bigg({t\cdot\left( 2\max\{ \norm{\nabla X}_u, \norm{\nabla Y}_u \}+1\right)}\bigg) \]
for $t\in[0,t_0]$. Therefore,
\[\ell(t)\leq \sqrt{e(0) + t\norm{Y-X}^2_u }\cdot \exp \bigg({t\cdot\left( \max\{ \norm{\nabla X}_u, \norm{\nabla Y}_u \}+\frac{1}{2}\right)}\bigg) \]
for $t\in[0,t_0]$. The distance between $u(0,t)$ and $u(1,t)$ is not greater than $\ell(t)$. Hence,
\[d\big(u(0,t), u(1,t)\big)\leq \alpha(t)\cdot e^{\beta(t)}\]
for all $t\in[0,t_0]$, where
\begin{align*}
    \alpha(t)&= \sqrt{\int_0^1 \norm{\partial_s u(s,0)}^2ds + t\cdot \norm{Y-X}^2_u},\\
    \beta(t)&= t\cdot \left( \max\{ \norm{\nabla X}_u, \norm{\nabla Y}_u \}+\frac{1}{2}\right).
\end{align*}
Thus we complete the proof.  \end{proof}

The next lemma applies Lemma~\ref{lem:gronwall} to Hamiltonian vector fields on the symplectization $M \times\R^+$ of a contact manifold $M$.

\begin{Lemma}\label{lem:gronapp}
Let $M$ be a closed contact manifold with a fixed contact 1-form and let $h^s: [0,1] \times M \to\R, s\in[0,1]$ be a smooth $s$-family of time-dependent contact Hamiltonians. Denote by $G^s: [0,1] \times M \times\R^+\to\R$ the Hamiltonian (on the symplectization of $M$) given by 
\[G^s_t(x,r):= r\cdot h_t^s(x).\]
Let $\mu:\R^+\to [0,1]$ be a smooth function such that $\abs{r\cdot \mu'(r)}<1$ for all $r\in\R^+$ and let $H^s: [0,1] \times M\times\R^+\to \R$ be the Hamiltonian defined by
\[H^s_t(x,r):= r\cdot h^{s\cdot \mu(r)}_t(x).\]
Let $g_M$ be a Riemannian metric on $M$ and let $g$ be the Riemannian metric on $M\times\R^+$ given by
\[g:= g_M+ \frac{dr\otimes dr}{r^2}.\]
Denote by $d$ the metric on $M\times\R^+$ induced by $g$ and by $\phi_t^s, \psi_t^s$ the Hamiltonian isotopies of $H^s$ and $G^s$, respectively. Then, there exist constants $k_1, k_2\in\R^+$ that are independent of the function $\mu$ such that
\[ d\left( \phi_t^s(x,r), \psi_t^{s\mu(r)}(x,r)\right)\leq k_1\cdot \sqrt{t\cdot \max_{r\in\R^+}\abs{r\mu'(r)}}\cdot e^{k_2\cdot t\cdot \left( \underset{r\in\R^+}{\max}\:\abs{r^2\mu''(r)} +1 \right)}\]
for $(x,r)\in M\times \R^+$ and $t\in\R.$
\end{Lemma}
\begin{proof} The proof will be divided into the following three steps. 
\medskip

\noindent \textbf{Step~1} (Global definiteness). This step proves that the Hamiltonians $H^s$ and $G^s$ have globally defined Hamiltonian isotopies. Since $g$ is a complete metric on $M \times \R^+$, in the view of Theorem~1.1 on page~179 in \cite{hirsch2012differential}, it is enough to show that the vector fields of $G^s$ and $H^s$ are bounded with respect to $g$. Denote by $Y^s_t$ the contact vector field on $M$ furnished by the contact Hamiltonian $h^s_t$. In other words, $Y_t^s$ is the vector field on $M$ characterized by
\begin{equation*}
    \alpha (Y^s_t)=-h^s_t \,\,\,\,\mbox{and}\,\,\,\, d\alpha(Y^s_t, \cdot )= dh^s_t - dh_t^s(R)\cdot \alpha,
\end{equation*}
where $\alpha$ is the fixed contact 1-form on $M$ and $R$ is the associated Reeb vector field. The Hamiltonian vector fields of $G^s$ and $H^s$ are given by
\begin{align*}
    & X^{G^s_t}(x,r)= Y_t^s(x) + rdh_t^s(x)(R)\partial_r,\\
    & X^{H^s_t}(x,r)= Y_t^{s\mu(r)}(x) + rdh_t^{s\mu(r)}(x)(R)\partial_r - sr\mu'(r)(\partial_s h)^{s\mu(r)}_t(x)\cdot R(x).
\end{align*}
Since we have the following estimations, 
\begin{align*}
    \norm{X^{G_t^s}(x,r)}\leq& \underset{s,t\in[0,1]}{\max_{x\in \Sigma}}\norm{Y_t^s(x)}_{g_\Sigma} + \underset{s,t\in[0,1]}{\max_{x\in \Sigma}} \abs{d h^s_t(x)(R)},\\
    \norm{X^{H^s_t}(x,r)}\leq& \underset{s,t\in[0,1]}{\max_{x\in \Sigma}}\norm{Y_t^s(x)}_{g_\Sigma} + \underset{s,t\in[0,1]}{\max_{x\in \Sigma}} \abs{d h^s_t(x)(R)} \\
    &+ \underset{s,t\in[0,1]}{\max_{x\in \Sigma}}\abs{(\partial_s h)^s_t(x)}\cdot \max_{x\in\Sigma} \norm{R(x)}_{g_\Sigma},
\end{align*}
the Hamiltonians $G^s$ and $H^s$ have globally defined Hamiltonian isotopies.

\medskip

\noindent \textbf{Step~2} ($C^1$ bounds). This step estimates the norms $\norm{\nabla X^{G^s_t}(x,r)}$ and $\norm{\nabla X^{H^s_t}(x,r)}$ of the linear maps $v \mapsto \nabla_v X^{G^s_t}(x,r)$ and $v \mapsto \nabla_v X^{H^s_t}(x,r).$ Since
\begin{align*}
    \nabla_v X^{G^s_t}(x,r) = & \nabla_v Y_t^s(x) + dh_t^s(x)(R)\cdot dr(v) \cdot \partial_r \\
    & +r d\left( dh_t^s(x)(R) \right)(v) \partial_r + r dh_t^s(x)(R)\nabla_v\partial_r,
\end{align*}
the triangle inequality implies
\begin{align*}
    \norm{\nabla X^{G^s_t}(x,r)} \leq & \norm{\nabla Y_t^s(x)} + \abs{dh_t^s(x)(R)}\cdot \frac{1}{r}\cdot \norm{dr} \\
    & + \norm{d\left( dh_t^s(x)(R) \right)}\cdot \norm{r\partial r} + \abs{r dh_t^s(x)(R)}\cdot \norm{\nabla \partial_r}\\
    = & \norm{\nabla Y_t^s(x)}_{g_M} + 2\cdot \abs{dh_t^s(x)(R)} + \norm{d\left( dh_t^s(x)(R) \right)}.
\end{align*}
In the equation above, we used $\norm{\nabla Y_t^s(x)}=\norm{\nabla Y_t^s(x)}_{g_M}$, $\norm{r\partial_r}=1$, $\norm{dr}=r$, and $\norm{\nabla\partial_r}=\frac{1}{r}$ (the latter follows from $\nabla_{\partial_r}\partial_r= -\frac{\partial_r}{r}$).  In particular, there exists a constant $C_1\in\R^+$ such that
\[\norm{\nabla X^{G^s_t}(x,r)}\leq C_1\]
for all $(x,r)\in M \times\R^+$, $s\in[0,1]$, and $t\in[0,1]$. 

To estimate $\norm{\nabla X^{H^s_t}(x,r)}$, the following observation will be useful. If $Z$ is a vector field on $M$, then
\[(x,r)\mapsto Z(x)\]
is a vector field on $M\times\R^+$ that is parallel along the curve $t\mapsto (x, r+t)$ for all $x\in M$ and $r\in\R^+$. The covariant derivative can be described in terms of parallel transport as
\[ \nabla_{v}Z=\lim_{h\to 0}\frac{\Pi(\gamma)_h^0Z(\gamma(h))-Z(\gamma)(0)}{h}, \]
where $\gamma$ is a smooth curve with $\gamma'(0)=v$ and $\Pi(\gamma)_h^0$ denotes the parallel transport along $\gamma$ from time $h$ to time $0$. This implies
\[ \nabla_{\partial_r} Y^{s\mu(r)}_t = \lim_{h\to 0}\frac{Y_t^{s\mu(r+h)}- Y_t^{s\mu(r)}}{h} = (\partial_s Y)_t^{s\mu(r)}\cdot s\cdot \mu'(r). \]
To ease the notation, denote $X^s_t:=X^{G^s_t}$. Since
\begin{align*}
    \nabla_{v+ a\partial_r}\left( X_t^{s\mu(r)}(x,r) \right) = & \left( \nabla_{v+ a\partial_r }X\right)_t^{s\mu(r)}(x,r) + a\cdot s\mu'(r)\cdot (\partial_s Y)_t^{s\mu(r)}(x) \\
    & + a\cdot d(\partial_s h)_t^{s\mu(r)}(x)(R)\cdot sr\mu'(r)\partial_r
\end{align*}
for $v\in T_x M$ and $a\in\R$, we can estimate
\begin{align*}
    \norm{\nabla\left(X_t^{s\mu(r)}(x,r)\right)} \leq & \norm{(\nabla X)_t^{s\mu(r)}(x,r)} + \abs{ sr\mu'(r) \cdot (\partial_s Y)_t^{s\mu(r)}(x) }\\
    & +  \abs{d(\partial_s h)_t^{s\mu(r)}(x)(R)\cdot s r\mu'(r)}\\
    \leq & \norm{(\nabla X)_t^{s\mu(r)}(x,r)} + \abs{(\partial_s Y)_t^{s\mu(r)}(x)} + \abs{d(\partial_s h)_t^{s\mu(r)}(x)(R)}.
\end{align*}
As a consequence, there exists a constant $C_2\in\R^+$ such that
\[ \norm{\nabla\left(X_t^{s\mu(r)}(x,r)\right)}\leq C_2 \]
for all $x\in M$, $r\in\R^+$, $s\in[0,1]$, and $t\in[0,1]$. Denote
\[ V_t^s(x,r):= X^{G_t^{s\mu(r)}}(x,t) - X^{H^s_t}(x,r) =  s r \mu'(r) (\partial_s h)_t^{s\mu(r)}(x)\cdot R(x). \]
For $v\in T_x M$ and $a\in\mathbb{R}$, we have
\begin{align*}
    \nabla_{v+a\partial_r} V_t^s(x,r) = & (v + a\partial_r)\left( sr\mu'(r) (\partial_s h)_t^{s\mu(r)}(x) \right)\cdot R(x) \\
    & +  sr\mu'(r) (\partial_s h)_t^{s\mu(r)}(x)\cdot \nabla_v R(x)\\
    = & sr\mu'(r) \left(d(\partial_s h)_t^{s\mu(r)}\right)(x) (v) \cdot R(x) + a s \frac{d}{dr} (r\mu'(r))(\partial_s h)_t^{s\mu(r)}(x) \cdot R(x)  \\
    & + a s^2 r \left(\mu'(r)\right)^2 \left( \partial^2_{ss}h \right)_t^{s\mu(r)}(x) \cdot R(x) +  sr\mu'(r) (\partial_s h)_t^{s\mu(r)}(x)\cdot \nabla_v R(x).
\end{align*}
Consequently, there exist constants $C_3, C_4\in\R^+$ such that
\[\norm{\nabla V_t^s(x,r)}\leq C_3 + C_4\cdot\abs{r^2\mu''(r)} \]
for all $x\in M$, $r\in\R^+$, $s\in[0,1]$, and $t\in[0,1]$. Hence, by the triangle inequality, 
\[\norm{\nabla X^{H^s_t}(x,r)}\leq C_2 + C_3 + C_4\cdot \abs{r^2\mu''(r)}\]
for all $(x,r)\in M\times\R^+$, $s\in[0,1]$, and $t\in[0,1]$. 

\medskip

\noindent \textbf{Step~3} (Final details). Now, we fix $s\in[0,1]$ and apply Lemma~\ref{lem:gronwall} to the vector fields $X^{H^s_t}$ and $X^{G_t^{s\mu(r)}}$. Since $X^{H^s_t}$ and $X^{G_t^{s\mu(r)}}$ are bounded (with respect to $g$), the time-dependent vector field
\[(x,r) \mapsto (1-a)\cdot X^{H^s_t}(x,r) + a\cdot Y^{G^{s\mu(r)}_t}(x,r)\]
is bounded as well. Consequently, the flow of this vector field is globally well defined for all $a\in[0,1]$. Since
\begin{equation*}
    \norm{X^{G_t^{s\mu(r)}}(x,r) - X^{H_t^s}(x,r)} \leq \max_{r\in\R^+}\abs{r\mu'(r)}\cdot \underset{s,t\in[0,1]}{\max_{x\in M}} \norm{(\partial_sh^s_t)(x)R(x)},
\end{equation*}
Lemma~\ref{lem:gronwall} implies there exist constants $k_1, k_2\in\R^+$ such that
\[d\left( \phi_t^s(x,r), \psi_t^{s\mu(r)}(x,r)\right)\leq k_1\cdot \sqrt{t\cdot \max_{r\in\R^+}\abs{r\mu'(r)}}\cdot e^{k_2\cdot t\cdot \left( \underset{r\in\R^+}{\max}\:\abs{r^2\mu''(r)} +1 \right)} \]
for $(x,r)\in M\times \R^+$ and $t\in[0,1].$ In fact, the constants $k_1$ and $k_2$ can be chosen to be independent of the function $\mu$ (that satisfies $\abs{r\mu'(r)}\leq 1$).
\end{proof}

The next lemma proves that one can interpolate between $r\cdot h^0$ and $r\cdot h^1$ in symplectization without creating 1-periodic orbits if $h^0$ and $h^1$ can be joined by a path of contact Hamiltonians that do not have any 1-periodic orbits.

\begin{Lemma}\label{lem:existsmu}
Let $M$ be a closed contact manifold with a fixed contact form and let $h^s: [0,1] \times M \to \R$, $s\in[0,1]$ be a smooth $s$-family of time-dependent contact Hamiltonians. Assume that $h^s$ has no 1-periodic orbits for all $s\in[0,1]$. Then, for every $a\in\R^+$, there exists a smooth function $\mu:\R^+\to [0,1]$ with the following properties
\begin{enumerate}
    \item the Hamiltonian
    \[H^s_t: M\times\R^+\to \R\quad:\quad (x,r)\mapsto r\cdot h^{s\cdot\mu(r)}(x)\]
    has no 1-periodic orbits for all $s\in[0,1]$;
    \item $\mu(r)=0$ for $r\in(0, a]$ and $\mu(r)=1$ for $r$ large enough.
\end{enumerate}
\end{Lemma}
\begin{proof}
Let $g_M$ be a Riemannian metric on $M$. Denote by $g$ the Riemannian metric on $M\times\R^+$ given by $g:= g_M+\frac{dr\otimes dr}{r^2}.$ Denote by $d_M$ and $d$ the metrics on $M$ and $M\times\R^+$ furnished by $g_M$ and $g$, respectively. The Pythagorean theorem implies
\[d_M(x,y)\leq d\big((x, r_x), (y, r_y)\big)\]
for all $x,y\in M$ and $r_x, r_y\in\R^+$. Denote by $\varphi^s$, $\phi^s$, and $\psi^s$ the isotopies of the contact Hamiltonian $h^s$, the Hamiltonian $H^s$, and the Hamiltonian on $M \times \R^+$ defined by $(x,r)\mapsto r\cdot h^s_t(x)$, respectively. In particular, $\psi^s_t(x,r)= \left(\varphi^s_t(x), f_t^s(x)\cdot r\right)$ for a certain positive smooth function $f_t^s:M\to\R^+$. Since $h^s$ has no 1-periodic orbits for all $s\in[0,1]$, the number
\[\varepsilon:=\inf_{x\in M, \,\, s\in[0,1]} d_M\left( x, \varphi^s_1(x) \right)\]
is positive. By Lemma~\ref{lem:gronapp}, there exists $\delta\in\R^+$ such that 
\[d\left( \phi_t^s(x,r), \psi^{s\mu(r)}_t(x,r) \right) <\frac{\varepsilon}{2}\]
if $\abs{r\mu'(r)}<\delta$ and $\abs{r^2\mu''(r)}<\delta$ for all $r\in\R^+$. Hence, the triangle inequality implies
\begin{align*}
    d\left((x,r), \phi^s_t(x,r)\right) \geq & d\left((x,r), \psi^{s\mu(r)}_t(x,r)\right) - d\left( \psi_t^{s\mu(r)}(x,r), \phi^s_t(x,r) \right)\\
    \geq & d_M\left(x, \varphi_t^{s\mu(r)}(x)\right) - \frac{\varepsilon}{2} \geq \frac{\varepsilon}{2}
\end{align*}
if $\abs{r\mu'(r)}<\delta$ and $\abs{r^2\mu''(r)}<\delta$ for all $r\in\R^+$. Therefore, it is enough to prove that for every $a, \delta\in\R^+$ there exists a smooth function $\mu:\R^+\to [0,1]$ such that
\begin{enumerate}
    \item $\mu(r)=0$ for $r\in(0,a],$
    \item $\mu(r)=1$ for $r$ big enough,
    \item $\abs{r\cdot\mu'(r)}<\delta$ for all $r\in\R^+$,
    \item $\abs{r^2\cdot\mu''(r)}<\delta$ for all $r\in\R^+$.
\end{enumerate}
The function $\mu$ can be constructed as follows. There exists a smooth function $f:\R\to[0,1]$ such that
\begin{enumerate}
    \item $f(t)=0$ for $t\in(-\infty, \ln a)$,
    \item $f(t)=1$ for $t$ large enough,
    \item $\abs{f'(t)}<\frac{\delta}{2}$ for all $t\in\R$,
    \item $\abs{f''(t)}<\frac{\delta}{2}$ for all $t\in\R$.
\end{enumerate}
Indeed, one can start with a function $\tilde{f}:\R\to[0,1]$ that satisfies the first two conditions, and then, define $f(t):=\tilde{f}(k\cdot t)$ for $k\in(0,1)$ sufficiently small. Now, take $\mu(r):=f(\ln r)$. By the construction, $\mu(r)=0$ for $r\in(0,a)$ and $\mu(r)=1$ for $r$ large enough. Additionally,
\[\abs{e^t\cdot \mu(e^t)}= \abs{f'(t)}<\frac{\delta}{2}\]
and
\begin{align*}
    \abs{e^{2t}\cdot \mu''(e^t)}&\leq \abs{e^{2t}\cdot \mu''(e^t) + e^t\cdot \mu'(e^t) } + \abs{e^t\cdot \mu'(e^t)}\\
    &= \abs{f''(t)} + \abs{f'(t)} < \delta
\end{align*}
for all $t\in\R$. Hence, $\abs{r\mu'(r)}<\delta$ and $\abs{r^2\mu''(r)}<\delta$ for all $r\in\R^+$. This shows that the function $\mu$ with the desired properties exists and finishes the proof.
\end{proof}

\vspace*{5mm}

\bibliographystyle{amsplain}
\bibliography{biblio_loop}
\vspace*{5mm}

\end{document}